\documentclass{amsart}

\usepackage{amsmath,amssymb}
\usepackage[hidelinks]{hyperref}

\newtheorem{theorem}{Theorem}[section]
\newtheorem{lemma}[theorem]{Lemma}
\newtheorem{proposition}[theorem]{Proposition}
\newtheorem{corollary}[theorem]{Corollary}

\theoremstyle{definition}
\newtheorem{definition}[theorem]{Definition}

\theoremstyle{remark}
\newtheorem{remark}[theorem]{Remark}

\newcommand{\R}{\mathbb R}
\newcommand{\ep}{\varepsilon}
\newcommand{\eps}{\varepsilon}

\newcommand\chb{{\mathcal Q}_s}
\newcommand{\Si}{\Sigma}
\newcommand{\Om}{\Omega}

\newcommand{\Oms}{\Omega\setminus\Sigma}

\newcommand{\Alo}{\mathcal{A}_L(\Omega)}
\newcommand{\Ain}{\Omega_i^{(n)}}

\newcommand{\haus}{\mathcal{H}^1}
\newcommand\meas{\mathop{\rm meas}}
\newcommand{\weak}{\rightharpoonup^*}
\newcommand{\prob}{\mathcal{P}(\overline{\Om})}

\newcommand{\la}{\lambda}
\newcommand{\lapz}{\Lambda_p}
\newcommand{\lan}{\lambda_p^{\sigma,\rho}}
\newcommand{\lap}{\lambda_{p}}

\newcommand{\Flp}{F_L}
\newcommand{\Fp}{F}
\newcommand{\Fli}{F_L^{\infty}}
\newcommand{\Fi}{F^{\infty}_\infty}
\newcommand{\Flpnew}{F_L^{(p)}}
\newcommand{\Fpnew}{F^{(p)}_\infty}
\newcommand{\Flinew}{F_L^{\infty}}

\DeclareMathOperator*{\esssup}{ess\,sup}
\DeclareMathOperator*{\essinf}{ess\,inf}

\newdimen\mex
\makeatletter
\def\niv{\mathrel{\hbox{\hglue -0.4\mex
\vrule \@height 1.4\mex \@width .14\mex
\vrule \@height .14\mex \@width .75\mex
\hglue -0.2\mex}}}
\makeatother

\title[Asymptotics of the first Laplace eigenvalue]{Asymptotics of the first Laplace eigenvalue with Dirichlet regions of
prescribed length}
\author{Paolo Tilli and Davide Zucco}

\address{Paolo Tilli, Dipartimento di Scienze Matematiche, Politecnico di Torino, Italy}
\email{paolo.tilli@polito.it}

\address{Davide Zucco, Scuola Internazionale Superiore di Studi Avanzati, Trieste, Italy}
\email{davide.zucco@sissa.it}

\begin{document}

\begin{abstract}
We consider the problem of maximizing the first eigenvalue of the
$p$-laplacian (possibly with non-constant coefficients) over a
fixed domain $\Omega$, with Dirichlet conditions along
$\partial\Omega$ and along a supplementary set $\Sigma$, which is
the unknown of the optimization problem. The set $\Sigma$, that
plays the role of a supplementary stiffening rib for a membrane
$\Omega$, is a compact connected set (e.g. a curve or a connected
system of curves) that can be placed anywhere in
$\overline{\Omega}$, and is subject to the constraint of an upper
bound $L$ to its total length (one-dimensional Hausdorff measure).
This upper bound prevents $\Sigma$ from spreading throughout
$\Omega$ and makes the problem well-posed. We investigate the
behavior of optimal sets $\Sigma_L$ as $L\to\infty$ via
$\Gamma$-convergence, and we explicitly construct certain
asymptotically optimal configurations. We also study the behavior
as $p\to\infty$ with $L$ fixed, finding connections with
maximum-distance problems related to the principal frequency of
the $\infty$-laplacian.
\end{abstract}

\maketitle


\paragraph{\textbf{AMS:} 35P15, 35P20, 49Q10}

\pagestyle{myheadings}
\thispagestyle{plain}

\section{Introduction}

Variational problems for the eigenvalues of the Laplace operator
have a long history, which dates back to 1877 when Lord Rayleigh
observed and conjectured that of all membranes with a given area,
the circle has the minimum principal frequency. Lord Rayleigh's
conjecture was proved many years later by G. Faber in 1923 and,
independently, by E. Krahn in 1924, and many other contributions
to similar problems were given ever since (see \cite{hen} for
other references on the subject).

Here we consider a new optimization problem for the first eigenvalue
of an isotropic elliptic operator with nonconstant coefficients in two dimensions,
where the unknown is the location of the Dirichlet condition.
The general setting consists of:
\begin{itemize}
    \item a bounded connected open set $\Omega\subset\R^2$, with Lipschitz boundary $\partial\Omega$
    (we do not assume that $\Omega$ is simply connected);
    \item two  functions $\rho,\sigma$, continuous over $\overline{\Omega}$ and strictly positive;
    \item a real number $p\geq 1$.
\end{itemize}
With these ingredients, for any open set $A\subset\Omega$ one can define
the first Dirichlet eigenvalue $\lan(A)$ through the variational formula
\begin{equation}\label{eq:eigen2}
\lan(A):= \inf_{u\in W^{1,p}_0(A) \atop u\not\equiv 0} \frac{\int_{A} \sigma(x)|\nabla u(x)|^pdx}{\int_{A}\rho(x)|u(x)|^p dx}
\end{equation}
(when $p>1$ the infimum is attained by the so called first eigenfunction,
while for $p=1$ it is not attained and is tightly related to the \emph{Cheeger constant}
see \cite{bucaco}). Note that, if $\rho=\sigma\equiv 1$, we obtain the first eigenvalue of the
$p$-laplacian, denoted for simplicity by $\lap$:
\begin{equation}\label{deflap}
\lap(A):= \inf_{u\in W^{1,p}_0(A) \atop u\not\equiv 0} \frac{\int_{A} |\nabla u(x)|^pdx}{\int_{A}|u(x)|^p dx}.
\end{equation}
For every $L>0$, as in \cite{butsan} we define the class of admissible sets
\begin{equation}
\label{alo}
\Alo:=\left\{\Sigma\subset\overline{\Omega}\,\,:\,\,\text{$\Sigma$ is a continuum, and $\haus(\Sigma)\leq L$}\right\}
\end{equation}
where ``continuum'' stands for ``connected, compact and non-empty set'' and
$\haus$ denotes the one-dimensional Hausdorff measure.

The variational problem we consider is
\begin{equation}\label{prob2}
 \max \left\{\lan(\Om\setminus\Si): \; \Sigma\in \Alo\right\},
\end{equation}
that is, to find the best location $\Sigma\in\Alo$ for an \emph{extra Dirichlet condition}
(in addition to that along $\partial\Omega$), in order to maximize the first eigenvalue $\lan$.
In particular, we shall study the behavior of the optimal sets, as $L\to\infty$, via $\Gamma$-convergence.

A possible physical interpretation of \eqref{prob2}, at least when $p=2$,
is the following. The domain
$\Omega$ represents an elastic structure (e.g. a membrane) in the plane
(with density $\rho$ and Young modulus $\sigma$) which is fixed
along its boundary $\partial\Omega$ and, as such, has a
fundamental frequency given by $\sqrt{\lan(\Omega)}$.
One is interested in augmenting and possibly  maximizing this
fundamental frequency, by fixing the membrane along a
supplementary curve (or system of curves) $\Sigma$ of given
length, which may be placed anywhere in $\overline{\Omega}$.
One may think of $\Sigma$ as a sort of stiffening rib,
to obtain a reinforced structure $\Omega\setminus\Sigma$.
Note
that not only the location, but also the shape of $\Sigma$ is
free (cf.~\cite{hakrku}, where the shape is a ball of given radius and
only its placement is to be optimized).
The role of the parameter $L$ is an upper bound to the total
resources available: if fixing the structure
along a system of curves $\Si$ has a cost proportional to its
total length $\haus(\Si)$, then $L$ is the maximum cost one is willing to
spend to reinforce the membrane.

Another problem, in the same spirit but with the compliance functional instead
of the first eigenvalue, has been treated in \cite{butsan}. Among the main differences
is the fact that \eqref{prob2} is not driven by
an \emph{a priori} given PDE, which reflects into the non-locality of the resulting
$\Gamma$-limit (see below).
The class of admissible sets $\Alo$ is also typical of the well-studied
\emph{average distance problems}, where the distance function to
$\Sigma$ is to be minimized. These
were first introduced by Buttazzo, Oudet and Stepanov in \cite{buoust, butste}, while
the asymptotics of minimizers was studied in \cite{mostil} (see also \cite{leme}).

\smallskip

As in \cite{buoust,butsan}, the geometric restrictions on $\Sigma$ entailed by \eqref{alo}
are useful in view of an existence result for
\eqref{prob2} (connectedness of $\Sigma$ can be relaxed to a bound on the number of
connected components, but some control is needed to prevent
$\Sigma$ from spreading throughout $\Omega$ and prejudice existence).
By classical results of Blaschke and Go{\l}ab (see~\cite{ambtil})  the space
$\Alo$ is compact in the Hausdorff metric, and the map
$\Sigma\mapsto \lan(\Omega\setminus\Sigma)$
is continuous (see \cite{sve} for $p=2$ and \cite{buctre} for the general
case). This leads to
\begin{theorem}[Existence]\label{teo:ex2}
For every $L>0$, there exists a maximizer in \eqref{prob2}. Moreover,
every maximizer $\Sigma$ satisfies $\haus(\Sigma)=L$.
\end{theorem}

The second claim follows from the elementary consideration that,
if  $\haus(\Sigma)<L$, then one could increase
$\lan(\Omega\setminus\Sigma)$ by attaching to $\Sigma$ some short
segments, thus reducing all the connected components of
$\Omega\setminus\Sigma$ (cf.~\cite{buoust,butsan}). Also note that
the \emph{minimization} problem analogous to  \eqref{prob2}
would be trivial, since $\lan(\Oms)\geq\lan(\Om)$ for every
$\Sigma\in\Alo$ and equality is achieved by any admissible $\Si$
hidden inside $\partial \Om$.

Contrary to average distance problems where regularity results are available
for the optimal sets $\Sigma$
(see~\cite{butste,santil}),
similar questions for \eqref{prob2} are open (except for Ahlfors regularity which we will
not discuss here).

In this paper we mainly focus on the asymptotic behavior of optimal sets
$\Si_L$ of problem $\eqref{prob2}$ as $L\to \infty$, with the goal
of studying the limit distribution of $\Si_L$ within
$\Om$. Of course, as $L$ increases, the optimal
configurations $\Sigma_L$ tend to saturate $\Omega$ (i.e. $\Sigma_L\to\overline{\Omega}$
in the Hausdorff distance),
and any information concerning the density (length
of $\Sigma_L$  per unit area in $\Omega$) is lost in the limit. To retrieve
this information we prove a $\Gamma$-convergence result,
in the space $\prob$ of
probability measures in $\overline{\Omega}$, identifying each admissible
$\Si\in\Alo$ with the probability measure
\begin{equation}\label{probability}
 \mu_\Si=\frac{\haus\niv\Si}{\haus(\Sigma)},
\end{equation}
where $\haus\niv\Si$ denotes the restriction of the one-dimensional Hausdorff measure to $\Sigma$
(this expedient is natural in this kind of problems, see \cite{mostil,butsan}). Thus,
our problem \eqref{prob2} becomes equivalent to the \emph{minimization}
of the functional
$\Flp:\prob\to [0,\infty]$ defined as
\begin{equation}\label{fl}
\Flp(\mu)=
\begin{cases}
\dfrac{L^p}{\lan(\Oms)} &\text{if  $\mu=\mu_\Si$ for some $\Si\in\Alo$,}\\[2mm]
\infty  &\text{otherwise.}
\end{cases}
\end{equation}
The scaling factor $L^p$, as we will see, is natural  as the maximum achieved in \eqref{prob2} grows
as $L^p$ for large $L$ (and, of course, rescaling does not
alter the original problem anyhow).
The  $\Gamma$-convergence result we will prove is then the following.
\begin{theorem}\label{teogam2}
As $L\to\infty$, the functionals $\Flp$ defined in \eqref{fl}
$\Gamma$-converge, with respect to the weak*
topology on $\prob$, to the functional $\Fp:\prob\mapsto [0,\infty]$ defined as
\begin{equation}\label{gammalimittwo}
\Fp(\mu):=\frac{1}{{\lapz}}\esssup_{x\in\Omega}
\frac{\rho(x)}{\sigma(x)f(x)^p}
\end{equation}
where
$f\in L^1(\Om)$ is the
density (Radon-Nikodym derivative) of $\mu$ with respect to the
Lebesgue measure, while $\lapz$ is the numerical constant
\[
{\lapz}:=(p-1)
\left(\frac{2\pi}{p\sin\left(\pi/p\right)}\right)^p
\quad\text{if $p>1$,}
\qquad
\Lambda_1:=2.
\]
\end{theorem}

\begin{remark}
The constant $\lapz$ (see~\cite{busedm,lindqvist})
is just
the first Dirichlet eigenvalue for the  $p$-laplacian in \emph{one
variable}, namely
\begin{equation}\label{defLambda}
\lapz = \inf_{u\in W^{1,p}_0(0,1) \atop u\not\equiv 0} \frac
{\int_0^1 |u'(z)|^p\,dz}{\int_0^1 |u(z)|^p\,dz}.
\end{equation}
If $p>1$, the infimum is attained
by the first eigenfunction $u_1$  that solves
the equation
\begin{equation}\label{eqei}
-(|u_1'|^{p-2}u_1')'={\lapz} |u_1|^{p-2}u_1,\quad
u_1(0)=u_1(1)=0.
\end{equation}
When $p=1$, the
infimum is not attained and $\Lambda_1=2$, as one can see letting $u$
approximate the characteristic function $\chi_{(0,1)}$.
See \cite{busedm,lindqvist} for more details.
\end{remark}

The $\Gamma$-limit functional $\Fp$ defined in \eqref{gammalimittwo}  has a unique minimizer. Indeed
\[
\min_{\mu\in\prob} \Fp(\mu)=\frac{1}{{\lapz}}
\min_{\begin{subarray}{c}f\geq 0 \\ \int_\Om f \leq
1\end{subarray}}\esssup_{x\in\Om}
\frac{\rho(x)}{\sigma(x)f(x)^p}=\frac{1}{{\lapz}}
\Big(\int_{\Om}\left(\frac{\rho(x)}{\sigma(x)}\right)^{1/p}dx\Big)^p
\]
achieved only when $\mu=\mu_\infty$, the absolutely continuous measure with density
\begin{equation}\label{eq:densf2}
f(x)=\frac{\left(\rho(x)/\sigma(x)\right)^{1/p}}{\int_{\Om}\left(\rho(y)/\sigma(y)\right)^{1/p}\,dy}
\end{equation}
(note that $\mu_\infty$ reduces to normalized Lebesgue measure, if $\rho$ and $\sigma$ are constant).
As the space $\prob$ is compact in the weak* topology, from standard $\Gamma$-convergence theory
(see~\cite{dalmaso}) we can recover the limiting distribution of the optimal sets $\Sigma_L$ for large $L$:
\begin{corollary}\label{cor:main2}
If $\Si_L$ is a maximizer of problem \eqref{prob2}, then as $L\to\infty$ the probability measures $\mu_{\Si_L}$ converge,
in the weak* topology of $\prob$, to the probability measure $\mu_\infty$, absolutely continuous with respect to the Lebesgue measure,
having the density given in \eqref{eq:densf2}. In particular, for every square $Q\subset\Omega$,
\begin{equation}\label{asopt}
\lim_{L\to\infty}\frac{\haus(\Sigma_L\cap Q)}{\haus(\Si_L)}=\int_Q f(x)\,dx
\end{equation}
and, moreover,
\begin{equation}
\label{asoptbis}
 \lim_{L\to \infty} \frac{L^p}{\lan(\Oms_L)}=\Fp(\mu_\infty)=\frac {\left(\int_{\Om}(\rho(x)/\sigma(x))^{1/p}dx\right)^p}{{\lapz}}.
\end{equation}
\end{corollary}

This corollary formalizes the ansatz that, in order to
maximize the principal frequency of a membrane with density $\rho$ and Young modulus $\sigma$,
it is convenient
to concentrate the stiffening rib $\Si$ in those regions with higher ratio $\rho/\sigma$, with
a density proportional to
$(\rho/\sigma)^{1/p}$. In particular,  for a homogeneous membrane
with constant $\rho$ and $\sigma$, it is convenient to distribute
$\Sigma$ with, roughly speaking, a constant ratio of length per
unit area in $\Omega$.

Moreover, the proof of Theorem~\ref{teogam2} is constructive: it turns
out that certain comb-shaped patterns
(see Definition~\ref{def:comb}),  periodically reproduced inside $\Omega$ at
different scales, can be used to build
examples of
asymptotically optimal sets (that is, those sets $\Sigma_L$ satisfying~\eqref{asoptbis}).

In section \ref{sec:max} we will also see that the
eigenvalue problems \eqref{prob2}, as $p\to\infty$ with $L$ being fixed,
converge to the so called \emph{maximum distance problem},
for which several qualitative results on the minimizers
have been proved in \cite{mipast} and \cite{paoste}.

Finally, let us mention that Theorem~\ref{teogam2} may be considered as the
``eigenvalue counterpart''
of the $\Gamma$-convergence results obtained
in \cite{mostil} and \cite{butsan} for
\emph{average distance} and \emph{compliance} problems.

\section{Estimates for the first eigenvalue under length constraints}

Throughout the paper,
we denote by $\mathop{d}(x,C)=min_{y\in C}|y-x|$ the distance
function to the set $C$, a generic closed subset of $\R^2$.
Moreover, we denote by $\meas(E)$ the two-dimensional Lebesgue
measure and by $\haus(E)$ the one-dimensional Hausdorff measure of
a Borel set $E\subset \R^2$. We will deal with the level sets of the
distance function, and in particular we need
the following result
proved in \cite{til} (see also Lemma 4.2 in \cite{mostil}).
\begin{lemma}\label{lemmaarea}
Fix $L>0$ and $\Si\in\Alo$. For $t\geq 0$ let
\begin{equation}\label{sublevel}
A_t=\{x\in\Om\,|\,\mathop{d}(x,{\Si\cup \partial \Om})< t\}
\end{equation}
be the sublevel set of the distance function  to
$\Si\cup\partial\Om$. If $\kappa$ is the number of connected components
of $\partial\Om$ and
\begin{equation}\label{deft}
 \overline{t}:=\dfrac{\meas(\Om)}
 {\left(\haus(\Si\cup\partial\Om)+\sqrt{\haus(\Si\cup\partial\Om)^2+(\kappa+1)\pi\meas(\Om)}\right)}
\end{equation}
is the positive root of the quadratic equation
$2\haus(\Si\cup\partial\Om)\overline{t}+(\kappa+1)\pi
\overline{t}^2=\meas(\Om)$, then for every $t\geq 0$
\begin{equation}
\label{www} \meas(A_t)\leq H(t):=
\begin{cases}
2\haus(\Si\cup\partial\Om)t+(\kappa+1)\pi t^2 \qquad &\text{if $t\leq\overline{t}$,}\\
\meas(\Om)\qquad &\text{if $t>\overline{t}$.}
\end{cases}
\end{equation}
\begin{remark}
The number of connected components of $\partial \Om$ is
necessarily finite,  since $\partial \Om$ is Lipschitzian
and compact. Similarly, also $\haus(\partial \Om)$ is finite.
\end{remark}

\begin{remark}\label{rema2}
From the definition of $\overline{t}$ in Lemma~\ref{lemmaarea} we
see that the function $H(t)$ in \eqref{www} is Lipschitzian and
increasing. Letting
$T=\max_{x\in\overline\Omega}\mathop{d}(x,{\Si\cup\partial\Om})$,
since $\overline{A_T}=\overline{\Omega}$ while
$H(t)<\meas(\Omega)$ for $t<\overline{t}$, we see that
\begin{equation}\label{lowerboundt}
0<\overline{t}\leq
T:=\max_{x\in\overline\Omega}\mathop{d}(x,{\Si\cup\partial\Om})
\end{equation}
and
\begin{equation}\label{HT}
\meas(A_T)=\meas(\Omega)=H(T).
\end{equation}
Moreover, $H'(t)\equiv 0$ for $t>\overline{t}$.
\end{remark}

\end{lemma}

We start by proving an upper bound for the first  eigenvalue
$\la_p(\Oms)$ (defined as in \eqref{deflap})
in terms of the length $\haus(\Si)$. As we
will see in Theorem \ref{theta}, this estimate is sharp
when $\haus(\Si)$ is large. Some of the techniques that we
use are refinements of those in \cite{til}, where a similar bound was proved for the
compliance functional.

\begin{theorem}\label{teotheta}
Let $\Omega\subset\R^2$ be a bounded, connected open set with a Lipschitz boundary $\partial\Omega$
made of $\kappa$ connected components.
For any $L$ and $\Si\in\Alo$, it holds
\begin{equation}\label{estb}
\lap(\Oms)\leq\frac{{\lapz}}{\left(2\overline{t}\right)^p}\left(1+\frac{(\kappa+1)\pi\overline{t}}{\haus(\Si\cup\partial\Om)}\right),
\end{equation}
where $\overline{t}$ is the number defined in \eqref{deft}.
\end{theorem}

\begin{proof}
By \eqref{deflap}, for any non-zero function $u\in  W_0^{1,p}(\Oms)$
we have
\begin{equation}\label{eqdim}
 \lap(\Oms)\leq \frac{\int_\Om |\nabla u(x)|^pdx}{\int_\Om |u(x)|^pdx},
\end{equation}
and a suitable choice of $u$
will lead to \eqref{estb}. More precisely, we choose $u$
depending on the distance function
\begin{equation}\label{functest}
u(x):=g(\mathop{d}(x,{\Si\cup\partial\Om})), \qquad x\in\Om,
\end{equation}
where $g:\R^+\to\R^+$ is $C^{1,1}$, increasing and concave, and
such that $g(0)=0$.  Then $u$ vanishes along $\Si\cup\partial\Om$,
is Lipschitzian and, in particular, it is an admissible function
for the Rayleigh quotient in \eqref{eqdim}.

We first estimate the numerator in \eqref{eqdim}. Since
$|\nabla\mathop{d}(x,{\Si\cup\partial\Om})|=1$ a.e., from the
coarea formula (see \cite{evagar}) we have
\begin{equation}\label{wwww}
\int_\Om |\nabla u(x)|^pdx = \int_\Om
g'(\mathop{d}(x,\Si\cup\partial\Om))^pdx= \int_0^{T} g'(t)^p
P(A_t,\Om) \,dt,
\end{equation}
where
$T=\max_{x\in\overline\Omega}\mathop{d}(x,{\Si\cup\partial\Om})$,
 $A_t$ is as in \eqref{sublevel} and $P(A_t,\Omega)$ is the perimeter of
$A_t$ in $\Om$ (see \cite{evagar} for more details on perimeters).
Recall that still from the coarea formula one has
$\meas(A_t)=\int_0^t P(A_t,\Om)dt$  for every $t>0$ and
hence
\[
P(A_t,\Om)=\frac{d}{dt}\meas(A_t)\quad \text{for a.e.
$t>0$}.
\]
Letting $G(t)=g'(t)^p$, as $G'(t)\leq 0$ by assumption, we can
integrate by parts in \eqref{wwww} and use \eqref{www}. Since
$\meas(A_0)=0$, we obtain
\[
\begin{split}
  &\int_\Om |\nabla u(x)|^pdx =-\int_0^{T}G'(t)\meas(A_t)\,dt +G(T)\meas(A_T)
\\
 \leq&-\int_0^{T}G'(t) H(t)\,dt +G(T)H(T).
\end{split}
\]
Since $H(0)=0$, we can integrate by parts the other way round,
thus obtaining
\begin{equation}\label{eqnum}
  \int_\Om |\nabla u(x)|^pdx \leq
\int_0^{T}g'(t)^p H'(t)\,dt=\int_0^{\overline{t}}g'(t)^p H'(t)\,dt
\end{equation}
where the last equality follows from Remark~\ref{rema2}

Similarly, letting $G(t)=g(t)^p$ and observing that now $G'(t)\geq
0$, recalling \eqref{www} and \eqref{HT} we can estimate from
below the denominator of \eqref{eqdim} as follows
\[
\begin{split}
  &\int_\Om |u(x)|^pdx
= \int_0^{T} g(t)^p P(A_t,\Om) \,dt =
-\int_0^{T}G'(t)\meas(A_t)\,dt +G(T)\meas(A_T)
\\
&\geq -\int_0^{T}G'(t) H(t)\,dt +G(T)H(T) = \int_0^{T}G(t)
H'(t)\,dt
 = \int_0^{\overline{t}}g(t)^p H'(t)\,dt.
\end{split}
\]
Now we plug the last estimate and \eqref{eqnum} into
\eqref{eqdim}: observing that
\[
2\haus(\Si\cup\partial\Om)\leq H'(t)\leq
2\haus(\Si\cup\partial\Om)+ 2(\kappa+1)\pi \overline{t},\quad t\in
(0,\overline{t})
\]
and changing variable $z=t/ (2\overline{t})$ in the two integrals, from
\eqref{eqdim} we find that
\begin{equation}\label{eqdim2}
 \lap(\Oms)\leq\left(1+\frac{(\kappa+1)\pi\overline{t}}{\haus(\Si\cup\partial\Om)}\right)
 \frac{1}{\left(2\overline{t}\right)^p}\frac{\int_0^{1/2}g'(z)^p\,dz}{\int_0^{1/2}g(z)^p\,dz},
\end{equation}
now valid for every $C^{1,1}$ function $g$ increasing and concave
on $[0,1/2]$, and such that $g(0)=0$. In fact, by a density
argument, we can relax $g\in C^{1,1}(0,1/2)$ to $g\in W^{1,p}(0,1/2)$. If
$p>1$, the first Dirichlet eigenfunction $u_1(z)$ of the
$p$-laplacian on $(0,1)$ is symmetric with respect to $y=1/2$
and,
from \eqref{eqei}
it follows that
$u_1$ is also increasing and concave on $[0,1/2]$. This means that
one can choose $g(z)=u_1(z)$ in \eqref{eqdim2} and this gives
\eqref{estb}, since the ratio of the two integrals then reduces to
$\lapz$.

Finally, if $p=1$, it suffices to let $g(z)=\min\{1,nz\}$ in \eqref{eqdim2}
and then let $n\to \infty$ (recall that $\Lambda_1=2$).
\end{proof}

\begin{remark}
It is clear from the proof that, in order to obtain \eqref{estb}, there
is nothing special with
$\Sigma\cup\partial\Omega$ except that this set supports the Dirichlet condition
associated with $\lambda_p(\Oms)$ through the function space $W^{1,p}_0(\Oms)$.
Indeed, the estimate still holds if one considers
the first eigenvalue with a Dirichlet condition prescribed
along any compact set $D\subset\overline{\Omega}$ such that $0<\haus(D)<\infty$,
having $\textsc{k}$ connected components. Of course, in this case, one has
to replace $\haus(\Sigma\cup\partial\Omega)$ with $\haus(D)$ and $\kappa+1$ with
\textsc{k}.
\end{remark}

\section{Proof of the $\Gamma$-convergence result}\label{sec:gamma}
In this section we will prove Theorem~\ref{teogam2},
first proving
the $\Gamma$-liminf and  the $\Gamma$-limsup inequalities up to a multiplicative constant
$\theta_p$,
defined as follows:
\begin{equation}\label{eq:theta}
  \theta_p:=\inf\left(\liminf_{n\to \infty} \frac{L_n^p}{\lap(Y\setminus\Sigma_n)}\right),
\end{equation}
where $Y=(0,1)^2$ is the unit square and
the infimum is over all sequences of numbers $L_n\to \infty $ and all sequences of sets $\Sigma_n\in{\mathcal A}_{L_n}$ such that
$\haus(\Sigma_n)=L_n$.
Then we will compute explicitly this constant, showing that it is the inverse of the first Dirichlet eigenvalue of the $p$-laplacian on the unit interval (Theorem \ref{theta}).

In the sequel
we will often
use the following well-known properties of the first Dirichlet eigenvalue:

\noindent - \emph{Monotonicity}: for any two bounded open sets $A\subset B$,
$\lan(A)\geq\lan(B)$.

\noindent - \emph{Splitting over connected components}: if $A$ can be written as $\bigcup A_i$ with pairwise disjoint open sets $A_i\not=\emptyset$
(e.g., its connected components), then $\lan(A)=\min_i \lan(A_i)$.

\noindent - \emph{Comparison with the homogeneous case}:
for any open set $D\subset\Omega$,
on comparing \eqref{eq:eigen2} and \eqref{deflap}
we have
\begin{equation}\label{analog}
\frac{\inf_D \sigma}{\sup_D \rho}\lambda_p(D)
\leq
\lan(D)
\leq
\frac{\sup_D \sigma}{\inf_D \rho}\lambda_p(D).
\end{equation}

\subsection{The $\Gamma$-liminf inequality}\label{gammainf}

We start proving that the $\Gamma$-liminf functional $\Flp$ is minorized by the limit functional $\Fp$ defined, up to $\theta_p$, by \eqref{gammalimittwo}. We shall use some of the ideas that were introduced in \cite{mostil}, see also \cite{butsan}.
\begin{proposition}
For every probability measure $\mu\in\prob$ and every sequence $\{\mu_L\}\subset\prob$ such that $\mu_L\weak\mu$, it holds
\begin{equation}\label{liminf2}
  \liminf_{L\to \infty} \Flp(\mu_L)\geq  \theta_p\esssup_{x\in\Om}\frac{\rho(x)}{\sigma(x) f(x)^p}.
\end{equation}
\end{proposition}
\begin{proof} Consider an arbitrary subsequence (still denoted by $\{\mu_L\}$ for simplicity,
but $L$ should be regarded as $L_n$ etc.) for which the
\[
\lim_{L\to \infty} \Flp(\mu_L)
\mathop{=}^{\text{by \eqref{fl}}}\lim_{L\to \infty}\frac{L^p}{\lan(\Omega\setminus\Sigma_L)}
\]
exists and is finite. By finiteness of the limit, recalling \eqref{fl}
we can assume that each $\mu_L$ has the form \eqref{probability} for some $\Sigma_L\in{\mathcal A}_L$, so that
\begin{equation}\label{rrr}
    \mu_L(E)=\frac{\haus(\Sigma_L\cap E)}{\haus(\Sigma_L)}\quad\text{for all Borel sets $E\subset\overline{\Omega}$.}
\end{equation}
Moreover, finiteness also entails that $\lambda_p(\Omega\setminus\Sigma_L)$ (being comparable with
$\lan(\Omega\setminus\Sigma_L)$) tends to infinity: hence, if we choose any
open square $Q\subset\Omega$, by monotonicity
also $\lambda_p\left(Q\setminus\Sigma_L\right)\to \infty$, and in particular
the distance function $\mathop{\rm d}(x,Q\setminus\Sigma_L)$
uniformly tends to zero (as $L\to \infty$) over $Q$
(otherwise  $Q\setminus\Sigma_L$ would contain a ball $B$ of radius bounded away from zero,
and $\lambda_p\left(Q\setminus\Sigma_L\right)$ would be bounded from above).
Hence, applying Lemma~\ref{lemmaarea} with $\Omega:=Q$ and $\Sigma:=\Sigma_L$, we see from
\eqref{lowerboundt} that $\overline{t}\to 0$ as $L\to \infty$, and by \eqref{deft}
this means that
\begin{equation}\label{dimtheta2}
\lim_{L\to \infty} \haus(\Si_L\cap Q)=\infty\quad\text{(for every open square $Q\subset\Omega$).}
\end{equation}
Now fix $\eps>0$,
and consider an open square $Q\subset\Omega$,  whose role is to localize the estimate on $\Flp$. From
 the monotonicity of $\lap$ and \eqref{analog} 
it follows that,
\begin{equation*}
\frac{L^p}{\lan(\Oms_L)}
\geq
\frac{L^p}{\lan(Q\setminus\Sigma_L)}
\geq
\frac{\inf_Q \rho}{\sup_Q \sigma} \frac{L^p}{\lap(Q\setminus \Si_L)}
\end{equation*}
which, using $L\geq \haus(\Sigma_L)$  and \eqref{rrr}, gives
\begin{equation}\label{dimliminf21}
\frac{L^p}{\lan(\Oms_L)}
\geq
\frac{\inf_Q \rho}{\sup_Q \sigma}\,
\frac {1} {\mu_L(Q)^p}\,
\frac{\haus(\Si_L\cap Q)^p}{\lap(Q\setminus\Sigma_L)}.
\end{equation}
From $\mu_L\weak\mu$, we have that $\limsup \mu_L(Q)\leq \mu\left(\overline{Q}\right)$ and hence
\begin{equation}\label{eq1}
  \liminf_{L\to \infty} \frac{1}{\mu_L(Q)^p} \geq  \frac{1}{\mu\left(\overline{Q}\right)^p+\ep\meas(Q) ^p},
\end{equation}
where the quantity $\ep\meas(Q) ^p$ serves to avoid vanishing denominator, for the moment.
Moreover, if $a$ is the side-length of $Q$ and $Y=a^{-1}Q$ is a unit square, by scaling
\begin{equation}\label{uuu}
 \frac{\haus(\Si_L\cap Q)^p}{\la_p(Q\setminus\Sigma_L)}= \frac{a^p\haus(a^{-1}\Si_L\cap Y)^p}{a^{-p}\la_p(Y\setminus a^{-1}\Si_L)}
 = \meas(Q) ^p\frac{\haus(a^{-1}\Si_L\cap Y)^p}{\la_p(Y\setminus a^{-1}\Si_L)}.
\end{equation}
Now let $\Sigma_n:=\partial Y\cup(a^{-1}\Sigma_L\cap Y)$, and observe that $\Sigma_n$ is connected,
since by \eqref{dimtheta2} $\Sigma_L$ must cross the boundary of $Q$.
Hence, using \eqref{dimtheta2} again, by translation invariance we can use \eqref{eq:theta}
and, from \eqref{uuu}, estimate
\begin{equation}\label{eq2}
 \liminf_{L\to \infty}\frac{\haus(\Si_L\cap Q)^p}{\lap(Q\setminus\Sigma_L)}\geq \theta_p\meas(Q) ^p.
\end{equation}
Now combining \eqref{eq1} and \eqref{eq2} with \eqref{dimliminf21},  we obtain the estimate
\[
\liminf_{L\to \infty}\Flp(\mu_L)= \liminf_{L\to \infty}\frac{L^p}{\lan(\Oms_L)}
\geq
\frac{\inf_Q \rho}{\sup_Q \sigma}\,
 \frac{\theta_p\meas(Q) ^p}{\mu\left(\overline{Q}\right)^p+\ep\meas(Q) ^p},
\]
for every open square $Q\subset\Omega$.
Thus, if $f\in L^1(\Omega)$ is the density of $\mu$
and
$x\in\Omega$ is a Lebesgue point for $f$, letting $Q$ shrink towards $x$,
from Radon-Nikodym Theorem
we find that
\[
\liminf_{L\to \infty}\Flp(\mu_L)
  \geq
  \theta_p \frac{\rho(x)}{\sigma(x)(f(x)^p+\ep)} \quad \text{for a.e. $x\in\Om$}.
\]
Finally, letting $\ep\downarrow 0$ one obtains \eqref{liminf2}.
\end{proof}

\subsection{The $\Gamma$-limsup inequality}\label{gammasup} As in \eqref{eq:theta}, we denote by $Y$ the unit square $(0,1)^2$.

\begin{lemma}\label{lemma1}
Given $\eps>0$, there exists a compact connected set $\Sigma\subset\overline{Y}$ such that
\begin{equation}\label{estim}
 \frac{\haus(\Sigma)^p}{\lap(Y\setminus\Sigma)}<(1+\eps) \theta_p
\end{equation}
and
\begin{equation}\label{stl}
\partial Y\subset\Sigma.
\end{equation}
\end{lemma}
\begin{proof}
According to how $\theta_p$ was defined in \eqref{eq:theta}, we can take a sequence
of sets $\{\Sigma_n\}$ satisfying $L_n:=\haus(\Sigma_n)\to\infty$,
$\Sigma_n\in{\mathcal A}_{L_n}(Y)$ and
\[
\frac {\haus(\Sigma_n)^p}{\lap(Y\setminus\Sigma_n)}<(1+\eps/2)\theta_p\quad\forall n\geq 1.
\]
If $n$ is large enough, then clearly
\[
\frac {\bigl(\haus(\Sigma_n)+\haus(\partial Y)+1/2\bigr)^p}{\lap(Y\setminus\Sigma_n)}<(1+\eps)\theta_p,
\]
and defining $\Sigma:=\Sigma_n\cup\partial Y\cup S$, where $S$ is any segment of length at most $1/2$
that connects $\Sigma_n$ to $\partial Y$,
we can guarantee \eqref{estim} and
\eqref{stl} (note that $\lap(Y\setminus\Sigma)\geq \lap(Y\setminus\Sigma_n)$ by monotonicity).
\end{proof}

We start proving the $\Gamma$-limsup inequality for a particular class of measures.

\begin{definition}\label{defstep2}
For $s>0$, let $\chb$ denote the collection of all those open squares $Q_i\subset\R^2$, with side-length $s$
and corners
on the lattice $(s{\mathbb Z})^2$, such that $Q_i\cap \Omega\not=\emptyset$.
We say that a probability measure $\mu\in\prob$ is \emph{fitted to $\chb$}
if $\mu$ is absolutely continuous, with a density $f(x)>0$ which is \emph{constant}
on each set of the form $Q_i\cap \Omega$ with $Q_i\in\chb$. In formulae,
\begin{equation}\label{mufit}
 d\mu=f(x) dx, \qquad f(x)= \sum_{i} \alpha_i \chi_{\Omega_i}(x), \quad \Omega_i=\Om\cap Q_i,\quad
 \chb=\{Q_i\}
\end{equation}
where the constants $\alpha_i>0$ satisfy (since $\mu(\Omega)=1$)
the normalization condition
\begin{equation}\label{eq:probmeas2}
\sum_{i} \alpha_i\meas(\Omega_i) = 1.
\end{equation}
\end{definition}

\begin{proposition}
If $\mu\in\prob$ is fitted to $\chb$ for some $s>0$, then for every
$\ep>0$ there exists a sequence $\{\mu_L\}$ in $\prob$ such that $\mu_L\weak \mu$ and
\begin{equation}\label{gammalimsup2}
 \limsup_{L\to \infty} \Flp(\mu_L)\leq (1+\ep) \theta_p \esssup_{x\in\Om}\frac{\rho(x)}{\sigma(x)f(x)^p}.
\end{equation}
\end{proposition}
\begin{proof}
Consider $\mu$ fitted to $\chb$, with the same notation as in \eqref{mufit}. As $\partial\Omega$ is Lipschitzian, by replacing (if necessary) $s$
with $s/2^k$ for some $k>1$ (thus keeping $\mu$ fitted to $\chb$), we may assume that $s$ is so small that
\begin{equation}\label{assu1}
\text{no connected component of $\partial\Omega$ is strictly contained in any square $Q_i\in\chb$.}
\end{equation}
Given a small $\eps>0$, we will construct the measures $\mu_L$ of the form \eqref{rrr} for suitable sets $\Sigma_L\in\Alo$.
We will call ``tile'' the set  $\Sigma$ obtained from Lemma~\eqref{lemma1}, satisfying \eqref{estim}
and \eqref{stl}. We define its "effective length" $L_e$ as the number
\begin{equation}\label{Le}
L_e:=\haus\left(\Sigma\cap [0,1)^2\right).
\end{equation}
Given $L$ large enough, the set $\Sigma_L$ is obtained through the following construction:
\begin{itemize}
    \item [(i)] fix a set $\Omega_i$ in \eqref{mufit} and scale down the tile $\Sigma$ to a factor $s/k_i$, where
the integer $k_i$ is defined as
\begin{equation}\label{defki}
k_i=k_i(L)=\left\lfloor \frac{s\alpha_i (L-\sqrt{L})}{L_e}\right\rfloor.
\end{equation}
    \item [(ii)] Put $k_i^2$ copies of the rescaled tile inside the closed
    square $\overline{Q_i}$ corresponding to $\Omega_i$, as to form a $k_i\times k_i$
    chalkboard, and intersect with $\Omega$ (the resulting set is contained in $\overline{\Omega_i}$).
    \item [(iii)] Repeat for each $\Omega_i$, and take the union. Finally, add $\partial\Omega$ to the resulting set.
\end{itemize}
Formally, if $l_i\in\R^2$ denotes the lower-left corner of the square $Q_i$, this construction amounts to defining
\[
\Sigma_L:=\partial\Omega\cup \left(\bigcup_i \bigcup_{0\leq m,n<k_i} \Sigma_{i,m,n}\cap\Omega\right),\quad
\Sigma_{i,m,n}=l_i+s(m/k_i,n/k_i)+(s/k_i)\Sigma.
\]
The main idea is to put several \emph{microtiles} $\Sigma_{i,m,n}$ side by side within each square
$\overline{Q_i}$, with a density (length per unit area) therein roughly proportional to $\alpha_i$, in such a way that
the total length is about $L$ (subtracting $\sqrt{L}$ in \eqref{defki} serves to save some $o(L)$ of length, to compensate
for $\haus(\partial\Omega)$ and boundary effects due to cut-out microtiles close to $\partial\Omega$).

Building on \eqref{assu1} and \eqref{stl}, a little thought reveals that the
set $\Sigma_L$ thus constructed is connected.

Each set $\overline{\Omega_i}$ contains a certain number $W_i(L)$ of \emph{whole} microtiles $\Sigma_{i,m,n}$
and, if $Q_i$ crosses $\partial\Omega$, also a certain number 
of \emph{incomplete} microtiles
$\Sigma_{i,m,n}\cap \Omega$
(those that have really been cut out
by intersection with $\Omega$, in step (ii) above).
As $\partial\Omega$ is Lipschitzian, however, and $\mathop{\rm diam}(\Sigma_{i,m,n})=O(1/L)$,
for large $L$ there are at most $C_1L$ incomplete microtiles,
where $C_1$ depends on $\mu$, $\haus(\Sigma)$ and $\partial\Omega$ but not on $L$. Moreover, as $\haus(\Sigma_{i,m,n})=O(1/L)$,
the incomplete tiles contribute to $\haus(\Sigma_L)$ by, at most, a constant length $C_2$ independent of $L$.
Hence, since clearly $W_{i}(L)\leq\meas(\Omega_i)k_i^2/s^2$,
\[
\begin{split}
\haus(\Sigma_L)&\leq \haus(\partial\Omega)+C_2+\sum_i W_i(L)\frac {s L_e}{k_i}
\leq
C_3+\sum_i \meas(\Omega_i)\frac {k_i L_e}{s}
\\
&\leq
C_3+
(L-\sqrt{L})\sum_i \meas(\Omega_i)\alpha_i
\leq L
\end{split}
\]
provided $L$ is large enough (recall \eqref{eq:probmeas2}). This shows that $\Sigma_L\in\Alo$ for large $L$, hence defining
$\mu_L$ as in \eqref{rrr}, from \eqref{fl} we obtain
\begin{equation}\label{lss}
 \limsup_{L\to \infty} \Flp(\mu_L)
 =
  \limsup_{L\to \infty} \frac{L^p}{\lan(\Omega\setminus\Sigma_L)}.
\end{equation}
Note that, in fact, $\haus(\Sigma_L)\sim L$ as $L\to\infty$. Indeed,
since $\overline{\Omega_i}$ contains $W_i(L)$ whole microtiles (each contributing an effective length $sL_e/k_i$)
and $W_i(L)\sim \meas(\Omega_i)k_i^2/s^2$ as $L\to\infty$, from \eqref{defki} and \eqref{eq:probmeas2}
\[
\haus(\Sigma_L)\geq
\sum_i W_i(L)\frac {s L_e}{k_i}
\sim
\sum_i \meas(\Omega_i)\frac {k_i L_e}{s}
\sim L\sum_i \meas(\Omega_i)\alpha_i
= L.
\]
Now, as the restriction of $\mu_L$ to each $\Omega_i$ is
a periodic homogenization of the same pattern with period $s/k_i$,
it is clear that the $\mu_L$ converge, as $L\to\infty$, to some
measure in $\prob$ which is fitted to $\chb$. More precisely,
recalling \eqref{mufit}
\[
\begin{split}
\mu_L\left(\overline{\Omega_i}\right)&=\frac{\haus\left(\Sigma_L\cap \overline{\Omega_i}\right)}{\haus(\Sigma_L)}
\sim
\frac{W_i(L)sL_e/k_i}
{L}
\sim
\frac{\bigl(\meas(\Omega_i)k_i^2/s^2\bigr)sL_e/k_i}
{L}\\
&=
\frac{\meas(\Omega_i)k_i L_e }
{sL}
\sim \meas(\Omega_i)\alpha_i=\mu(\Omega_i),
\end{split}
\]
and we see that in fact $\mu_L\mathop{\rightharpoonup}^* \mu$ as $L\to\infty$.

To estimate $\lan(\Omega\setminus\Sigma_L)$ in \eqref{lss}, fix $L$ and observe that $\Omega\setminus\Sigma_L$
consists, by construction, of several small connected components (at least one for each microtile
$\Sigma_{i,m,n}$, due to \eqref{stl}). Therefore, since the first Dirichlet eigenvalue splits over connected components, for a suitable triplet $i,m,n$,
\[
\lan(\Omega\setminus\Sigma_L)=\lan(D),\quad D:=\Omega\cap (Y_{i,m,n}\setminus\Sigma_{i,m,n})
\]
where $Y_{i,m,n}$ is the open square
of side $s/k_i$, contained in $Q_i$,
that frames $\Sigma_{i,m,n}$.
Therefore, using \eqref{analog},
\[
\frac{L^p}
{\lan(\Omega\setminus\Sigma_L)}
=
\frac{L^p}
{\lan(D)}
\leq
\frac{\sup_D \rho}{\inf_D \sigma}\,
\frac{L^p}
{\lambda_p(D)}.
\]
Moreover, by monotonicity, scaling, \eqref{estim} and \eqref{Le},
\[
\lambda_p(D)\geq \lambda_p(Y_{i,m,n}\setminus\Sigma_{i,m,n})
=
\frac {k_i^p}{s^p} \lambda_p(Y\setminus\Sigma)
\geq
\frac{\bigl(k_i  \haus(\Sigma)\bigr)^p}{s^p(1+\eps)\theta_p}
\geq
\frac{(k_i L_e)^p}{s^p(1+\eps)\theta_p}
\]
which plugged into the previous estimate gives
\begin{equation}\label{jjj}
\frac{L^p}
{\lan(\Omega\setminus\Sigma_L)}
\leq
\frac{\sup_D \rho}{\inf_D \sigma}\,\left(
\frac{sL}{k_i L_e}\right)^p
(1+\eps)\theta_p.
\end{equation}
Now $D\subset \Omega_i$ and, by \eqref{defki}, $\mathop{\rm diam}(D)=O(1/L)$. Hence,
by positivity and uniform continuity of $\rho$ and $\sigma$ over $\Omega$,
\[
\frac{\sup_D \rho}{\inf_D \sigma}
\leq
(1+\delta_L)\sup_{\Omega_i} \rho/\sigma,
\quad\text{and}\quad
\left(
\frac{sL}{k_i L_e}\right)^p
\leq
\frac {1+\delta_L}{\alpha_i^p}
\]
where $\delta_L$ is independent of $i$ and tends to zero as $L\to\infty$.
Therefore, \eqref{gammalimsup2} follows from \eqref{lss} and \eqref{jjj}, taking the limsup there,
and observing that
\[
\esssup_{x\in\Om}\frac{\rho(x)}{\sigma(x)f(x)^p}
=
\max_j
\left(
\frac 1 {\alpha_j^p} \sup_{\Omega_j} \rho/\sigma\right),
\]
as $f(x)$ is piecewise constant according to \eqref{mufit}.
\end{proof}

Finally, we prove the density in energy of those measures $\mu\in\prob$ that
are fitted to $\chb$ for some $s>0$.
Then, by a general result of $\Gamma$-convergence theory \cite{dalmaso},
the $\Gamma$-limsup inequality \eqref{gammalimsup2} will be established for every probability measure $\mu\in\prob$.

\begin{proposition}\label{prop:denen2}
For every $\mu\in\prob$ there exists a sequence $\{\mu_n\}\subset \prob$ such that
every $\mu_n$ is fitted to $\chb$ for some $s>0$,
$\mu_n\weak\mu$ and
\begin{equation}
\label{claim1}
\limsup_{n\to \infty} \Fp(\mu_n)\leq \Fp(\mu).
\end{equation}
\end{proposition}\newcommand\hn{^{(n)}}
\begin{proof}
Consider an arbitrary measure $\mu\in\prob$.
Keeping the notation of Definition~\ref{defstep2}, we construct $\mu_n\in\prob$, fitted to ${\mathcal Q}_{1/n}$, as follows.
By analogy with \eqref{mufit}, we set
\begin{equation*}
 d\mu_n=f_n(x) dx, \qquad f_n(x)= \sum_{i} \alpha_{i}\hn \chi_{\Omega_{i}\hn}(x), \quad \Omega_{i}\hn=\Om\cap Q_i\hn,\quad
 {\mathcal Q}_{1/n}=\{Q_i\hn\}
\end{equation*}
where the numbers $\alpha_i\hn$ are chosen as to satisfy the conditions
\begin{equation}
\label{conda}
\frac{\mu\left(\Ain\right)}{\meas\left(\Ain\right) }
\leq
\alpha_i\hn
\leq
\frac{\mu\left(\,\overline{\Ain}\,\right)}{\meas\left(\Ain\right) },\quad
\sum_i \alpha_i\hn\meas\left(\Omega_i\hn\right)=1.
\end{equation}
Note that, for fixed $n$, the $\{\Omega_i\hn\}$ are pairwise disjoint while their closures $\left\{\overline{\Omega_i\hn}\right\}$
cover $\overline{\Omega}$. Then the double inequality above means that $\mu_n$ is a sort of
sampling of $\mu$, and it is easy to see that $\mu_n\mathop{\rightharpoonup}^*\mu$ as $n\to\infty$.

Let $f\in L^1(\Omega)$ be the density of $\mu$ with respect to the Lebesgue measure. Recalling \eqref{gammalimittwo},
passing to reciprocals we see that \eqref{claim1} reduces to
\begin{equation}\label{dimendens11}
\liminf_{n\to \infty} \bigl( \essinf_{x\in\Om} g(x)f_n(x)\bigr)\geq \essinf_{x\in\Om}g(x)f(x),\quad
\text{where $g(x):=\frac{\sigma(x)^{1/p}}{\rho(x)^{1/p}}$.}
\end{equation}
Define the quantity $\tau_n=\min_i \bigl(\inf_{\Omega_i\hn} g/\sup_{\Omega_i\hn} g\bigr)$ and observe that,
as $\mathop{\rm diam}(\Omega_i\hn)=O(1/n)$, by uniform continuity and positivity of $\sigma,\rho$ over $\overline{\Omega}$,
$\tau_n\to 1$ as $n\to\infty$.
To estimate the first $\essinf$ in \eqref{dimendens11}, as the $\partial \Omega_i\hn$ are Lebesgue-negligible, we
can restrict ourselves to consider $x\in \Omega_i\hn$ for some $i$.
Then, using \eqref{conda},
\[
g(x)f_n(x)=
g(x)\alpha_i\hn
\geq
\frac{g(x)\mu\left(\Ain\right)}{\meas\left(\Ain\right) }
\geq
\frac{g(x)\int_{\Omega_i\hn}f(y)\,dy
}{\meas\left(\Ain\right)}
\geq
\frac{\tau_n
\int_{\Omega_i\hn}g(y)f(y)\,dy}
{\meas\left(\Ain\right)}
\]
and hence, from the arbitrariness of $i$,
\[
\essinf_{x\in\Om} g(x)f_n(x)\geq \tau_n \essinf_{x\in\Om}g(x)f(x).
\]
Taking the liminf and using $\tau_n\to 1$, one obtains \eqref{dimendens11} as claimed.
\end{proof}

\subsection{Computation of $ \theta_p$ and optimal sequences}\label{sec:theta}
In this section we prove that the constant $\theta_p$, defined by \eqref{eq:theta}, is the
inverse of the first Dirichlet eigenvalue of the $p$-laplacian on the unit interval.
To this purpose, we define the following class of admissible sets.

\begin{definition}\label{def:comb}
Let $\overline{Y}=[0,1]^2$ be the closed unit square and let $n\geq 1$ be an integer. We define the
set  $C_n\subset\overline{Y}$
(called \emph{comb configuration})
as the union of $n+1$ equispaced vertical segments of length one, a distance of $1/n$ apart, together with the lower base of $\overline{Y}$
(the role of the latter is to make $C_n$ connected, see Figure~\ref{fig:combgrid}).
\end{definition}

\begin{theorem}\label{theta}
Recalling \eqref{defLambda} and \eqref{eq:theta}, there holds
\[
\theta_p= \frac{1}{{\lapz}}.
\]
Moreover, the constant $\theta_p$ is achieved, in \eqref{eq:theta}, when $\Sigma_n$ is the comb-structure
$C_n$ of Definition~\ref{def:comb}.
\end{theorem}

\begin{proof}
Consider any sequence of sets  $\Sigma_n\subset\overline{Y}$ such that $L_n:=\haus(\Sigma_n)\to\infty$, as described
after \eqref{eq:theta}. Applying Theorem~\ref{teotheta} with $\Omega:=Y$,  $\Sigma:=\Sigma_n$
(and $\kappa=1$ as $\partial Y$ is connected), we find the
bound \eqref{estb}, namely
\begin{equation}\label{aaa}
\lap(Y\setminus\Sigma_n)\leq\frac{{\lapz}}{2^p\overline{t}_n^p}\left(1+\frac{2\pi\overline{t}_n}{\haus(\Sigma_n\cup\partial Y)}\right),
\end{equation}
where
$\overline{t}_n$ is defined as $\overline{t}$ in \eqref{deft}. Note that,
since $\haus(\Sigma_n)=L_n\to\infty$,
\begin{equation}\label{bbb}
\overline{t}_n
=
\dfrac{1}{\left(\haus(\Sigma_n\cup\partial Y)+\sqrt{\haus(\Sigma_n\cup\partial Y)^2+2\pi}\right)}
\sim \frac{1}{2L_n} \quad \text{as $n\to \infty$}.
\end{equation}
Plugging  \eqref{aaa} into \eqref{eq:theta} and using \eqref{bbb},
from the arbitrariness of $\Sigma_n$ one obtains that $\theta_p\geq 1/{\lapz}$.

To prove the opposite inequality, let $C_n$ be the comb structure of Definition~\ref{def:comb}.
Note that $\haus(C_n)=n+2$, hence $C_n\in{\mathcal A}_{n+2}(Y)$. Moreover, the set $Y\setminus C_n$ is the union of $n$
rectangles of size $1/n\times 1$, hence
\begin{equation}\label{eigenrect}
\lambda_p(Y\setminus C_n)=\lambda_p(R_n),\quad R_n=(0,1/n)\times (0,1).
\end{equation}
Even though $\lambda_p(R_n)$ is not known explicitly, a lower bound is obtained by relaxing the boundary condition,
from Dirichlet on the whole $\partial R_n$ to Dirichlet on the two long sides of $R_n$ (and Neumann on the short, horizontal sides).
With these boundary conditions,
it is well known that the first eigenvalue coincides with the corresponding eigenvalue in one variable, on the interval $(0,1/n)$.
Thus, if $W$ is the subspace of those functions $w\in W^{1,p}(R_n)$ with null trace at $x=0$ and at $x=1/n$,
recalling \eqref{defLambda}
\[
\lambda_p(R_n)\geq
\inf_{w\in W\atop w\not\equiv 0}
\frac{\int_0^{1}\int_0^{1/n} |\nabla w(x,y)|^p\,dxdy}
{\int_0^{1}\int_0^{1/n} |w(x,y)|^p\,dxdy}
=
\inf_{u\in W^{1,p}_0(0,1/n) \atop u\not\equiv 0} \frac
{\int_0^{1/n} |u'(x)|^p\,dx}{\int_0^{1/n} |u(x)|^p\,dx}
=
n^p\lapz,
\]
that is $\lambda_p(Y\setminus C_n)\geq n^p\Lambda_p$. Now, if we choose $\Sigma_n=C_n$
(and $L_n=n+2$) in \eqref{eq:theta},
we find the optimal upper bound
\[
 \theta_p\leq\liminf_{n\to \infty}
 \frac{(n+2)^p}{\lap(Y\setminus C_n)}\leq \liminf_{n\to \infty}
 \frac{(n+2)^p}{n^p\lapz}
=\frac{1}{{\lapz}}.\qedhere
\]
\end{proof}

\begin{remark}
If $p=2$ or $p=1$, $\lap(R_n)$ in \eqref{eigenrect} is known explicitly.
More precisely, when $p=2$ it is well known that $\lambda_2(R_n)=\pi^2(n^2+1)$.
Moreover, when
 $p=1$ the first Dirichlet eigenvalue $\lambda_1(R_n)$ is just the Cheeger constant
of $R_n$, namely
\[
h(R_n)=\frac{4-\pi}{1+1/n-\sqrt{(1-1/n)^2+\pi/n}}=\Big(1+1/n+\sqrt{(1-1/n)^2+\pi/n}\Big)n
\]
(see \cite{kawfri,kawlac} for more details).
\end{remark}

\begin{figure}
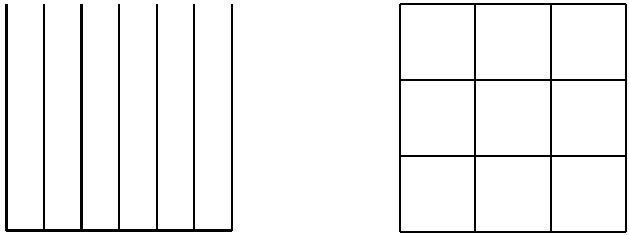

\begin{center}
\setlength\tabcolsep{0pt}\renewcommand\arraystretch{0}
\newcommand\tassello{\mbox{\rule[0.5cm]{0.5cm}{0pt}}}
\newcommand\griglia{\mbox{\rule[1cm]{1cm}{0pt}}}
\begin{tabular}{|c|c|c|c|c|c|}
\tassello & \tassello & \tassello & \tassello & \tassello & \tassello \\
\tassello & \tassello & \tassello & \tassello & \tassello & \tassello \\
\tassello & \tassello & \tassello & \tassello & \tassello & \tassello \\
\tassello & \tassello & \tassello & \tassello & \tassello & \tassello \\
\tassello & \tassello & \tassello & \tassello & \tassello & \tassello \\
\tassello & \tassello & \tassello & \tassello & \tassello & \tassello \\ \hline
\end{tabular}
\hspace{2cm}
\begin{tabular}{|c|c|c|c|c|c|}
\hline
\griglia & \griglia & \griglia \\ \hline
\griglia & \griglia & \griglia \\ \hline
\griglia & \griglia & \griglia \\ \hline
\end{tabular}
\end{center}
\caption{A comb-shaped configuration as opposite  to a grid structure.}\label{fig:combgrid}
\end{figure}

\begin{remark}\label{rem:conjecture}
We may call ``asymptotically optimal'' (for the unit square $Y$)
those sequences of admissible configurations $\Sigma_n$
(such as the comb configurations $C_n$) that achieve the infimum
$\theta_p$ in \eqref{eq:theta}. Other examples of asymptotically optimal
configurations are provided by \emph{oblique} comb structures,
that is, the intersection of $Y$ with a (thicker and thicker)
family of equispaced parallel lines (plus $\partial Y$ to make the structure connected):
the reason for asymptotic optimality is that (much like the vertical combs $C_n$)
these sets disconnect $Y$ into the union of (approximate) thin rectangles, and the
first eigenvalue of a thin rectangle is mainly governed by its short side (a detailed
proof would follow the same lines as the proof concerning $C_n$).

We point out, however, that for large length $L$
a comb configurations is \emph{strictly}
more performant (at least when $p=2$)
than a grid structure of about the same length (see Figure \ref{fig:combgrid}).
Indeed, while for $C_n$ the ratio $\haus(C_n)^2/\lambda_2(Y\setminus C_n)$ is
about $1/\Lambda_2=1/\pi^2$,
one can easily check that, replacing $C_n$ with a grid structure of about the same
length, the new ratio would approach $2/\pi^2$, hence a comb structure is twice more
performant than a grid structure.

Remarkably, the same comb structures are asymptotically optimal also for
average-distance problems (as proved in~\cite{mostil}) and for
the compliance optimization
(this was conjectured by Buttazzo and Santambrogio in~\cite{butsan} and later
proved in~\cite{til}).
\end{remark}

\section{The asymptotics as $p\to \infty$ and maximum distance problems}\label{sec:max}

In this section we investigate problem \eqref{prob2} as $p$ tends to $\infty$
(with fixed $L$),
showing that it converges to the problem
\begin{equation}\label{maxi}
\max \Big\{\frac{1}{\displaystyle \max_{x\in\Om}\mathop{d}(x, \Si \cup \partial \Om)}  : \; \Sigma\in \Alo\Big\}.
\end{equation}
This is not surprising,
since for every bounded domain $D$
\begin{equation}\label{lemmai}
\lim_{p\to \infty} \lap(D)^{1/p} = \frac{1}{\displaystyle \max_{x\in D}\mathop{d}(x,\partial D)}
\end{equation}
(and the right-hand side can be taken as the definition of $\lambda_\infty(D)$, the
principal frequency of the "$\infty$-laplacian", see \cite{julima}).
In other words, problem \eqref{prob2} in the limiting case $p=\infty$ reduces
to the so-called \emph{maximum distance problem}
\begin{equation}\label{mini}
\min \{ \max_{x\in\Om} \mathop{d}(x, \Si \cup \partial \Om): \; \Si\in\Alo\},
\end{equation}
that is, to search for those configurations $\Si\in\Alo$ that minimize
the radius of the largest ball that fits in $\Oms$ (see~\cite{paoste}).

In view of letting $p\to\infty$ in \eqref{prob2},
since now $L$ is fixed, scaling by $L^p$ as in \eqref{fl} would
be pointless: the proper normalization, suggested by
\eqref{lemmai}, is raising the eigenvalue to the power $1/p$. Moreover,
 it is no longer necessary to
work in the space of probability measures, as the set $\Alo$ provides the
natural common domain for the relevant functionals.

The precise $\Gamma$-convergence result is then the following.
\begin{theorem}\label{teoFlinew}
Fix $L>0$. As $p\to \infty$, the functionals $\Flpnew:\Alo\mapsto(0,\infty)$ defined for $p>1$ by
\[
 \Flpnew(\Si):=
 \frac{1}{\lan(\Oms)^{1/p}}
\]
$\Gamma$-converge, with respect to the Hausdorff distance on $\Alo$, to the $\Gamma$-limit
\[
 \Flinew(\Si):=
 \max_{x\in\Om}\mathop{d}(x, \Si \cup \partial \Om).
\]
\end{theorem}

\begin{proof} Since the functions $\rho,\sigma$ are uniformly positive and bounded,
by \eqref{analog}
we may assume
that $\rho,\sigma\equiv 1$, and work with $\lap$ in place of $\lan$. We start with the $\Gamma$-liminf inequality,
proving that for every $\Si\in\Alo$ and every
sequence $\{\Si_p\}\subset\Alo$ such that $\Si_p\to\Si$ in the Hausdorff distance, it holds
\begin{equation}\label{liminfinftyp22}
  \liminf_{p\to \infty} \Flpnew(\Si_p)
  \geq
\max_{x\in\Om}\mathop{d}(x, \Si \cup \partial \Om).
\end{equation}
For a fixed $\Sigma\in\Alo$ and  a sequence $\{\Sigma_p\}$ converging in the Hausdorff distance to $\Sigma$, choose a number $r>0$
such that $r<\max_{x\in\Om}\mathop{d}(x, \Si \cup \partial \Om)$. From the Hausdorff convergence of $\{\Si_p\}$ to $\Si$,
we see that $r<\mathop{d}(x,\Si_p\cup\partial \Om)$, and hence there exists a ball $B_r$ of radius $r$ such that $B_r\subset \Oms_p$,
provided $p$ is large enough. Therefore, by monotonicity of $\lap$ and \eqref{lemmai}
\[
 \liminf_{p\to \infty} \Flpnew(\Si_p)
 \geq
 \liminf_{p\to \infty} \frac 1{\lap(B_r)}
=
\max_{x\in B_r} \mathop{d}(x,\partial B_r)=r,
\]
and letting $r\to \max_{x\in\Om}\mathop{d}(x,\Si\cup\partial \Om)$ we obtain
\eqref{liminfinftyp22}. Finally, the $\Gamma$-limsup inequality follows
immediately
from the pointwise convergence of $\Flpnew$ to $\Flinew$, i.e. from \eqref{lemmai}.
Indeed, given $\Si\in\Alo$ one can define the constant sequence $\Sigma_p:=\Sigma$, which gives
\[
\limsup_{p\to \infty} \Flpnew(\Si_p)=
\lim_{p\to \infty} \Flpnew(\Si)=
 \Flinew(\Si). \qedhere
\]
\end{proof}

\begin{remark}
As a consequence of this $\Gamma$-convergence result and the compactness of the
space $\Alo$ with respect to the Hausdorff convergence, we get the stability of
the maximizers $\Si_p$, as $p$ converges to $\infty$. If $\Si_p$ is a maximizer
of problem \eqref{prob2} and $p \to \infty$, then, up to subsequences, the
sets $\Si_p$ converge in the Hausdorff distance to a
minimizer $\Si_\infty$ of problem \eqref{mini}. Moreover
\[
 \lim_{p\to \infty} \lan(\Oms_p)^{1/p}=\frac{1}{\displaystyle \max_{x\in\Om}\mathop{d}(x, \Si_\infty \cup \partial \Om)}.
\]
\end{remark}
The next $\Gamma$-convergence result is  the analogue of Theorem \ref{teogam2},
for the case $p=\infty$.
\begin{theorem}\label{tgammai}
As $L\to\infty$ the functionals  $\Fli:\prob\to [0,\infty]$ defined as
\begin{equation}\label{fli}
\Fli(\mu)=
\begin{cases}
\displaystyle L\max_{x\in\Om} \mathop{d}(x, \Si \cup \partial \Om)
&\text{if $\mu=\mu_\Si$ for some $\Si\in\Alo$,}\\
\infty  &\text{otherwise}
\end{cases}
\end{equation}
$\Gamma$-converge, with respect to the weak*  topology of $\prob$, to the functional $\Fi:\prob\mapsto
[0,\infty]$ defined by
\begin{equation}\label{gammalimiti}
\Fi(\mu)=\frac{1}{2}\esssup_{x\in\Omega} \frac{1}{f(x)},
\end{equation}
where $f\in L^1(\Om)$ is the density  of $\mu$ with respect to the Lebesgue measure.
\end{theorem}

\begin{proof} One can adapt, with only minor changes and several
simplifications, the proof of Theorem~\ref{teogam2},
the details are omitted. The only relevant change concerns the analogue of the constant $\theta_p$
and its computation (the analogue of Theorem~\ref{theta}).
For $p=\infty$, we adapt \eqref{eq:theta}
by defining
\begin{equation}\label{thetainf}
 \theta_\infty:=\inf\left\{\liminf_{n\to \infty} L_n \max_{x\in Y}  \mathop{d}
 (x, \Si_n \cup \partial Y)\right\},\quad Y=(0,1)\times(0,1),
\end{equation}
where, as in \eqref{eq:theta}, the infimum is
over all sequences $L_n\to \infty $ and all sequences of sets $\Sigma_n\in{\mathcal A}_{L_n}$ such that
$\haus(\Sigma_n)=L_n$. To complete the proof, we will show that $\theta_\infty=1/2$.

Consider two sequences $L_n$, $\Sigma_n$ as described above. Applying Lemma~\ref{lemmaarea}
with $\Omega=Y$ (hence $\kappa=1$) and $\Sigma=\Sigma_n$, we can use \eqref{lowerboundt} with
$\overline{t}$ defined as in \eqref{deft}. Since clearly $\haus( \Sigma_n\cup \partial Y)
\leq L_n+4$, \eqref{lowerboundt} gives
\[
L_n \max_{x\in Y} \mathop{d}(x, \Si_n \cup \partial Y)\geq
\frac{L_n}{(L_n+4)+\sqrt{(L_n+4)^2+2\pi}}.
\]
Letting $n\to\infty$, from the arbitrariness of $\Sigma_n$ and $L_n$, we
see from \eqref{thetainf} that $\theta_\infty\geq 1/2$.

On the other hand, choosing $\Si_n=C_n$ (the comb-shaped structure of Definition~\ref{def:comb})
and $L_n=\haus(C_n)=n+2$,
we clearly have
\[
L_n \max_{x\in Y} \mathop{d}(x, \Si_n \cup \partial Y)=\frac{L_n}{2n}=\frac{n+2}{2n},
\]
and letting $n\to\infty$ we see from \eqref{thetainf} that $\theta_\infty\leq 1/2$. Thus,
$\theta_\infty=1/2$.
\end{proof}

\begin{remark}
The $\Gamma$-limit functional
$\Fi$ has a unique minimizer, given by normalized Lebesgue measure over $\Omega$.
As a consequence (cf.~Corollary~\ref{cor:main2}),
if $\Si_L$  is a maximizer of problem \eqref{maxi}, as $L \to \infty$
 the probability measures $\mu_{\Si_L}$ converge in the weak* topology
to the uniform measure $dx/\meas(\Om)$ (the minimizer of $\Fi$).
\end{remark}

Finally, for completeness,
we prove that
the functionals defined in \eqref{gammalimittwo}, after renormalization,
$\Gamma$-converge to the functional in \eqref{gammalimiti}.
\begin{theorem}\label{teoFi2}
As $p\to \infty$ the functionals $\Fpnew$ defined by
\[
 \Fpnew(\mu)=\frac{1}{{\lapz}^{1/p}}\esssup_{x\in\Om} \frac{\rho(x)^{1/p}}{\sigma(x)^{1/p}f(x)},\qquad
 \mu\in\prob
\]
$\Gamma$-converge, in the weak* topology on $\prob$, to the $\Gamma$-limit $\Fi$ defined
by \eqref{gammalimiti}.
\end{theorem}
\begin{proof}
Since for all $\mu\in\prob$ we have $m_p \Fi(\mu)\leq \Fpnew(\mu)\leq M_p \Fi(\mu)$ where
\[
m_p=\frac 2 {\Lambda_p^{1/p}}\min_{x\in \overline\Omega}\frac{\rho(x)^{1/p}}{\sigma(x)^{1/p}},
\quad
M_p=\frac 2 {\Lambda_p^{1/p}}\max_{x\in \overline\Omega}\frac{\rho(x)^{1/p}}{\sigma(x)^{1/p}},
\]
and the constants $m_p,M_p\to 1$ as $p\to\infty$, it suffices to prove that the sequence
of functionals $\{\Fi\}_p$ (independent of $p$)
$\Gamma$-converges to $\Fi$ itself. But this is true, since
$\Fi$ (already obtained as a $\Gamma$-limit in $\prob$ by
Theorem~\ref{tgammai}) is a~fortiori lower semicontinuous (see Prop.~6.8 and Rem.~4.5 in~\cite{dalmaso}).
\end{proof}

An overall picture of these $\Gamma$-convergence results is given in the following
commutative diagram, where the first line is an equivalent formulation
(see~\cite{dalmaso}, Prop.~6.16)
of Theorem~\ref{teogam2}.
\bigskip

\begin{center}
\begin{tabular}{ccc}
$\displaystyle \Sigma \mapsto  \frac{L}{(\lan(\Oms))^{1/p}}$ &
$\xrightarrow[\text{(Thm \ref{teogam2})}]{\displaystyle\quad L\to\infty\quad}$ &
$\displaystyle \mu \mapsto  \frac{1}{{\lapz}^{1/p}}\esssup_{x\in\Om}\frac{\rho(x)^{1/p}}{\sigma(x)^{1/p}f(x)}$ \\[5mm]
$\mathop{}^{\displaystyle p\to\infty}\limits_{\text{(Thm \ref{teoFlinew})}}
 \Bigg\downarrow
$   & &
$\quad\quad\quad\qquad \Bigg\downarrow
\mathop{}^{\displaystyle p\to\infty}\limits_{\text{(Thm \ref{teoFi2})}}$ \\[5mm]
$\displaystyle \Si \mapsto  L \max_{x\in\Om}\mathop{d}(x, \Si \cup \partial \Om)$ &
$\xrightarrow[\quad\text{(Thm \ref{tgammai})}]{\displaystyle\quad L\to\infty\quad}$ &
$\displaystyle \mu \mapsto \frac{1}{2}\esssup_{x\in\Om}\dfrac{1}{f(x)}$
\end{tabular}
\end{center}

\end{document}